\documentclass{amsart}
\usepackage{graphicx}
\usepackage{etex}
\usepackage{bm}
\usepackage{amssymb}
\usepackage[stretch=10,shrink=10]{microtype}
\usepackage{tikz}
\usepackage{enumerate}
\usepackage{theoremref}
\usepackage{comment}
\usepackage[showframe=false, margin = 1.5 in]{geometry}
\usepackage{pgfplots}
\usepackage{hyperref}

\def\polhk#1{\setbox0=\hbox{#1}{\ooalign{\hidewidth
    \lower1.5ex\hbox{`}\hidewidth\crcr\unhbox0}}}

\newcommand{\dist}{\overset{dist}{=}}

\DeclareMathOperator{\lk}{lk}
\DeclareMathOperator{\st}{st}

\newcommand{\Z}{\mathbb Z}

\usepackage{color}

\usepackage{microtype}
\usepackage{mathrsfs}

\newcommand{\ER}{{Erd\H{o}s--R\'{e}nyi }}

\newcommand{\fsc}{\Delta} 
\newcommand{\fcc}{X} 
\newcommand{\vsupp}{\text{vsupp}}

\theoremstyle{plain}
\newtheorem{thm}{Theorem}[section]
\newtheorem{lem}[thm]{Lemma}

\newtheorem{define}[thm]{Definition}

\theoremstyle{remark}

\newenvironment{pfofthm}[1]
{\par\vskip2\parsep\noindent{\em Proof of Theorem\ #1. }}{{\hfill
$\square$}
\par\vskip2\parsep}

\tikzset{every tree node/.style={align=center}}
\usepackage{url}

\def\path{{\mathscr P}}

\usepackage{chngcntr}
\usepackage{apptools}
\AtAppendix{\counterwithin{thm}{section}}

\begin{document}

\title[Collapsibility of Random Clique Complexes]{Collapsibility of Random Clique Complexes}
\author{Greg Malen}
\begin{abstract}
We prove a sufficient condition for a finite clique complex to collapse to a $k$-dimensional complex, and use this to exhibit thresholds for $(k+1)$-collapsibility in a sparse random clique complex. In particular, if every strongly connected, pure $(k+1)$-dimensional subcomplex of a clique complex $\fcc$ has a vertex of degree at most $2k+1$, then $\fcc$ is $(k+1)$-collapsible. In the random model $X(n,p)$ of clique complexes of an \ER random graph $G(n,p)$, we then show that for any fixed $k\ge 0$, if $p=n^{-\alpha}$ for fixed $1/(k+1)<\alpha <1/k$, then a clique complex $\fcc\dist X(n,p)$ is $(k+1)$-collapsible with high probability. 
\end{abstract}

\maketitle

\section{Introduction}
The study of random topological spaces has become an incredibly important tool in a wide range of research. Of particular interest is the study of phase transitions for structural and topological properties of discretized spaces. These results were classicly studied for 1-dimensional random graphs, most notably in the \ER random graph $G(n,p)$, the probability space of graphs on $n$ vertices where each edge appears independently with probability $p$. Viewing graphs as simply 1-dimensional topological spaces, the search for analogous results in related high-dimensional models has become increasingly relevant to modern research.

Here we consider a natural extension of $G(n,p)$, the random model $X(n,p)$. This is the probability space of clique complexes on $n$ vertices where each edge appears independently with probability $p$. A clique complex, also known as a flag complex, is the maximal simplicial complex supported on a given set of edges, so this is equivalent to taking a graph sampled from $G(n,p)$ and adding in all simplices whose edge sets appear in full. If $\fcc$ is a complex sampled from $X(n,p)$, we write $\fcc\dist X(n,p)$, and then we say that $\fcc$ has property $A$ \emph{with high probability} if $P(\fcc \in A)\rightarrow 1$ as $n\rightarrow \infty$. Property $A$ is said to have a sharp threshold if there is some $p_A$ for which

$$\lim_{n\rightarrow\infty}P(X\in A)=\left\{\begin{array}{ll}
0 & \text{ if } p\leq(1-\epsilon)p_A\\
1 & \text{ if } p\geq(1+\epsilon)p_A\end{array}\right\}.$$
\\

Various topological features of this model, such as its homology and fundamental group, have been studied extensively in the so-called sparse regime in ~\cite{babson}, ~\cite{CFH15}, ~\cite{hkp12}, ~\cite{clique}, ~\cite{kahlesharp}, and ~\cite{Meckes1}. Our focus is on the collapsibility of $X(n,p)$. Collapses are an extremely useful tool in studying simplicial complexes, as they provide a method of reducing both the number of simplices and the overall dimension, making the computation of topological invariants far more tractable. Simplicial collapses are also homotopy equivalences, so the reduced complex has the same topological invariants as the original. We focus specifically on the property of being $k$-collapsible. A complex is said to be \emph{$k$-collapsible} if all faces of dimension at least $k$ can be collapsed, and so the complex is homotopy equivalent to a $(k-1)$-dimensional complex. 

Similar collapsibility thresholds were studied previously in another important extension of $G(n,p)$, the random $d$-dimensional model $Y_d(n,p)$. Collapsibility in this model was initially studied in ~\cite{ALLM13} and ~\cite{CCFK12}. In \cite{AL16} Aronshtam and Linial showed that the lower bound for $d$-collapsibility in ~\cite{ALLM13} is indeed tight, thus establishing a sharp threshold. 

In $X(n,p)$, we exhibit a lower bound for a $(k+1)$-collapsibility threshold. We use $(k+1)$-collapsibility, as opposed to $k$-collapsibility, in order to align this work with the property of a random clique complex being $k$-dimensional and the corresponding discussion of the Bouquet of Spheres conjecture in Section \ref{kspheres}. Our main result is the following.

\begin{thm}
\label{thm:collapse}
Fix an integer $k\geq 0$ and any $\epsilon>0$, and let $$p\leq n^{-1/(k+1)-\epsilon}.$$ 
If $\fcc \dist X(n,p)$, then with high probability $\fcc$ is $(k+1)$-collapsible.
\end{thm}

Homologically speaking, this improves on an earlier result that $H_{j}(X; \Z) = 0$ for $j \ge k+1$ ~\cite{clique}. It was also shown in ~ \cite{clique} that if $p=n^{-\alpha}$ for a fixed $1/(k+1) < \alpha < 1/k$, then $H_k(\fcc;\Z)\neq 0$. Therefore $\fcc$ is not $k$-collapsible in this range, and Theorem \ref{thm:collapse} is optimal with respect to the dimension. Since there cannot be torsion in top dimensional homology, it also rules out torsion in dimension $k$. And for $1/(k+1) < \alpha < 1/k$, the dimension of $\fcc$ is either $2k$ or $2k+1$ with high probability, so this further asserts that a sparse random clique complex collapses to middle dimension. The $k=1$ case of Theorem \ref{thm:collapse} was done in earlier work by Costa, Farber and Horak, appearing as Theorem A in \cite{CFH15}. 
\\

\section{Background and Definitions}

We first recall a few standard definitions and establish some notation. A simplicial complex $\fsc$ is a \emph{clique complex} if it is the maximal simplicial complex supported on its edge set, where every complete subgraph in its 1-skeleton is filled in with a simplex. Clique complexes are thus entirely determined by their edge sets, and they are general in the sense that every simplicial complex is homeomorphic to a clique complex, via barycentric subdivision. 

The \emph{faces} of a simplicial complex $\fsc$ are merely the simplices it contains, and for a face $\sigma\in\fsc$ with $\dim(\sigma)=k$, we call $\sigma$ a $k$-face. For a vertex $v \in \fsc$, we define the \emph{star of $v$ in $\fsc$} to be the set of faces which contain $v$, $$\text{st}_{\fsc} (v) : =\{ \tau \in \fsc \colon v\in\tau\}.$$\\ Similarly, the \emph{link of $v$ in $\fsc$} is the simplicial complex made up of the faces of $\fsc$ to which which $v$ can be added, $$\lk_{\fsc}(v)=\{\sigma\in\fsc : \sigma\cup\{v\}\in\fsc\}.$$\\
These two definitions generalize directly to the star and link of any face of $\fsc$, but will primarily be used herein with respect to vertices. It is straightforward that if $\fcc$ is a clique complex, then $\lk_{\fcc}(v)$ is also a clique complex for any vertex $v\in\fcc$. 

A $k$-dimensional simplicial complex $\fsc$ is said to be \emph{pure} if all of the maximal faces of $\fsc$ are $k$-dimensional. A pure $k$-dimensional simplicial complex $\fsc$ is then \emph{strongly connected} if for every pair of $k$-faces $\sigma$ and $\tau$ in $\fsc$, there is a sequence of $k$-faces in $\fsc$, $\sigma=\sigma_0,\sigma_1,\ldots,\sigma_j=\tau$, such that $\sigma_i \cap \sigma_{i+1}$ is a $(k-1)$-face for every $i$. These conditions are natural to study, as a minimal representative of a $k$-dimensional homology class of a simplicial complex is precisely a closed, pure $k$-dimensional, strongly connected subcomplex.

Let $\fsc$ be a simplicial complex with $\sigma,\tau\in \fsc$ such that $\tau$ is maximal in $\fsc$, $\sigma\subset\tau$, and $\sigma$ is not contained in any other maximal faces. Then $\sigma$ is called a \emph{free face} of $\tau$, or more generally of $\fsc$, and a \emph{simplicial collapse of the interval $[\sigma,\tau]$ in $\fsc$} is the removal of all faces $\eta\in \fsc$ such that $\sigma\subseteq\eta\subseteq\tau$. It is well known that simplicial collapses are homotopy equivalences.

Finally, for a subcomplex $S \subset \fsc$ we use $\vsupp(S)$ to denote the \emph{vertex support} of $S$, i.e.\ the set of vertices contained in faces of $S$. And we use the standard notation $\fsc^i$ for the $i$-skeleton of $\fsc$, i.e. $\fsc^i=\{\eta\in\fsc : \dim(\eta)\le i\}$. \\

\section{Proof of Theorem \ref{thm:collapse}}\label{tcl}

The proof of Theorem \ref{thm:collapse} will follow from Theorem \ref{thm:colla} and Lemma \ref{lem:coll_prob}, which respectively provide a sufficient condition for $(k+1)$-collapsibility of a deterministic clique complex, and show that for $p$ in the given range, this condition is satisfied with high probability.

\begin{thm} \label{thm:colla}
Fix $k\geq0$. Let $\fcc$ be a finite clique complex such that every strongly connected, pure $(k+1)$-dimensional subcomplex $S \subseteq \fcc$ contains at least one vertex $v$ with $\deg_S(v) \le 2k+1$. Then $\fcc$ is $(k+1)$-collapsible.
\end{thm}

For convenience, for any $k\ge 0$ we refer to subcomplexes which are both pure $k$-dimensional and strongly connected as the \emph{relevant $k$-subcomplexes} of $\fcc$. Note that if $\sigma \in \fcc$ is a face with $\dim(\sigma)\geq k+1$, then $\sigma^{k+1}$ is a relevant $(k+1)$-subcomplex of $\fcc$. So each face of dimension at least $k+1$ is supported on precisely one \emph{maximal} relevant $(k+1)$-subcomplex. 

To prove Theorem \ref{thm:colla}, it is then convenient to perform the collapses independently on each maximal relevant $(k+1)$-subcomplex. First we employ a closure operation to include all high-dimensional faces which need to be collapsed, and then we examine the intersections of two distinct such subcomplexes.

\begin{define} \label{flag:cl}
For a simplicial complex $\fsc$, define the flag closure $\overline{\fsc}=\fsc\cup\{\sigma:\sigma^{1}\subseteq\fsc\}$. This is equivalent to taking the clique complex of $\fsc^1$.
\end{define}

This closure preserves the 1-skeleton and vertex degrees of $\fsc$, so $\deg_{\fsc}(v)=\deg_{\overline{\fsc}}(v)$ for every $v\in\fsc$. Furthermore, if $\fcc$ is a clique complex with $\{S_1, \ldots, S_m\}$ the set of maximal relevant $(k+1)$-subcomplexes of $\fcc$, then the set $\left\{\overline{S_1},\ldots, \overline{S_m}\right\}$ forms a partition of the faces in $\fcc$ of dimension at least $k+1$.

A useful property of simplicial collapses is that the collapse of any interval $[\sigma, \tau]$ can be factored into a set of \emph{elementary} collapses of the form $[\eta_1,\eta_2]$ with $\dim(\eta_2)=\dim(\eta_1)+1$. Thus we can remove faces of dimension at least $k+1$ by only collapsing faces of dimension at least $k$. In particular, we must be sure that $k$-faces in the intersection of two maximal relevant subcomplexes do not interfere with performing collapses in each of the subcomplexes independently.

Before taking the closures, pure dimensionality forces any $k$-face in such an intersection to be a connecting bridge, so for $S_1$ and $S_2$ both maximal, $\dim\left(S_1\cap S_2\right)\le k-1$. In the closure, however, there can be arrangements where $k$-faces which are maximal in $\fcc$ are included in $\overline{S_i}$. Lemma \ref{int:dim} ensures that any such $k$-face is only contained in higher dimensional faces in one of the $\overline{S_i}$, so it can be collapsed there independently of collapses being performed in any distinct relevant subcomplexes. 

\begin{lem} \label{int:dim}
Let $\fcc$ be a clique complex with $S_1,S_2\subset \fcc$ distinct, maximal, relevant $(k+1)$-subcomplexes. Then $\dim\left(\overline{S_1}\cap \overline{S_2}\right)\le k$, and if $\sigma$ is a $k$-dimensional face in the intersection, then $\sigma$ is a maximal face in at least one of $\overline{S_1}$ and $\overline{S_2}$.
\end{lem}

\begin{proof} [Proof of Lemma \ref{int:dim}]
Suppose $\eta$ is a $(k+1)$-face in $\overline{S_1}\cap \overline{S_2}$. Then also $\eta\in S_1\cap S_2$, since by maximality all $(k+1)$-faces are included in some $S_i$ before applying the closure. But then $\eta$ forms a bridge between $S_1$ and $S_2$, and $S_1\cup S_2$ is strongly connected by transitivity, contradicting maximality. Hence $\dim\left(\overline{S_1}\cap \overline{S_2}\right)\le k$.

Now suppose that $\sigma$ is a $k$-face in the intersection which is not a maximal face in either $\overline{S_1}$ or $\overline{S_2}$, and let $\eta_1\in \overline{S_1}$ and $\eta_2\in \overline{S_2}$ be $(k+1)$-faces containing $\sigma$. As in the above, we have that in fact $\eta_1\in S_1$ and $\eta_2\in S_2$, and $\eta_1\neq\eta_2$. But then $\sigma\in S_1\cap S_2$, and it forms a bridge between $\eta_1$ and $\eta_2$, so by transitivity $S_1 \cup S_2$ is strongly connected, again contradicting maximality. Hence $\sigma$ must be maximal in at least one of $\overline{S_1}$ and $\overline{S_2}$.
\end{proof}

So for a $k$-face $\sigma\in \fcc$, $\text{st}_{\fsc}(\sigma)$ is contained in at most one of the $\overline{S_i}$. Therefore collapses which remove faces of dimension at least $k$ can all be performed independently in the various $\overline{S_i}$.

In the induction step in the proof of Theorem \ref{thm:colla}, it will be necessary to perform collapses in the link of a vertex. Lemma \ref{liftcollapse} provides a mechanism for lifting these to collapses of the correct dimension in $\fcc$.

\begin{lem} \label{liftcollapse}
Let $\fcc$ be a clique complex, with $S \subseteq \fcc$ a maximal relevant $(k+1)$-subcomplex. For a vertex $v\in S$, if $\sigma, \tau\in \lk_{\overline{S}}(v)$ such that $\sigma$ is a free face of $\tau$ and $\dim(\sigma)\geq k-1$, then $\sigma \cup \{v\}$ is a free face of $\tau \cup \{v \}$ in $\fcc$.
\end{lem}

\begin{proof}[Proof of Lemma \ref{liftcollapse}]
Since $\dim(\sigma)\geq k-1$, we have that $\dim(\sigma \cup \{v\} )\geq k$. And $\sigma \cup \{v\}$ is contained in $\tau \cup \{v\}$, so it is not maximal in $S$. Thus by Lemma \ref{int:dim}, all higher dimensional faces of $\fcc$ containing $\sigma \cup \{v\}$ are contained in $\overline{S}$, and are not in any other maximal relevant $(k+1)$-subcomplex. Hence\\
\begin{align*}
\{\eta\in \fcc:\sigma \cup \{ v \} \subseteq \eta \cup \{ v\} \}\subseteq \lk_{\overline{S}}(v).\\
\end{align*} 
But $\sigma$ is a free face of $\tau$ in $\lk_{\overline{S}}(v)$, so any such $\eta$ must also be contained in $\tau$. Therefore $\tau \cup \{v\}$ is maximal in $\fcc$, and $\sigma \cup \{ v \}$ is a free face of $\tau \cup \{v \}$ in $\fcc$.  
\end{proof}

Lemma \ref{lk:nicely} then follows as an immediate corollary of Lemma \ref{liftcollapse}. 

\begin{lem} \label{lk:nicely}
Let $\fcc$ be a clique complex, with $S \subseteq \fcc$ a maximal relevant $(k+1)$-subcomplex. Let $v \in S$ be a vertex such that $\lk_{\overline{S}}(v)$ is $k$-collapsible. Then $\fcc$ collapses to a complex in which $v$ is not contained in any $(k+1)$-dimensional faces of $\overline{S}$.
\end{lem}

\bigskip

We are now ready to prove Theorem \ref{thm:colla}. We begin by applying Lemma \ref{lk:nicely} in the specific context where $v$ is a vertex of sufficiently small degree. Upon removing all high-dimensional faces above $v$, we will need to find the new maximal relevant $(k+1)$-subcomplexes of the remaining complex, and also narrow our attention to subcomplexes of $\fcc$ which are still clique complexes. When $k\ge 3$, for instance, we are removing faces without changing the 1-skeleton, so $\fcc$ itself will no longer be a clique complex. 

\bigskip

\begin{proof}[Proof of Theorem \ref{thm:colla}]

We proceed by induction on $k$, and for a fixed $k$ we will further use induction on a uniform upper bound $l\ge|\vsupp(S)|$ for every maximal relevant $(k+1)$-subcomplex $S\subseteq \fcc$.\\

{\bf Base Case for $k$:} 
$k=0$. Let $\fcc$ be a finite clique complex which satisfies the conditions of the theorem for $k=0$. By definition, a strongly connected, pure 1-dimensional complex is merely a connected graph with at least one edge. Thus the degree condition of the theorem requires that every nontrivial connected subgraph contains a vertex of degree 1. This is equivalent to requiring that every nontrivial connected subgraph is acyclic, and so $\fcc$ is a forest. Any vertex of degree 1 can be collapsed along with its incident edge, and thus it is straightforward that a forest collapses to a set of points.\\

{\bf Inductive Step for $k$:} 
Now fix $k\ge1$, suppose that the theorem holds for $k-1$, and let $\fcc$ be a finite clique complex such that every relevant $(k+1)$-subcomplex $S\subseteq \fcc$ has at least one vertex $v$ with $\deg_S(v) \le 2k+1$.\\

{\bf Base Case for $l$:} 
$l=k+2$. Note that $k+2$ is the least number of vertices required to support $(k+1)$-dimensional faces. Suppose $|\vsupp(S)|\le k+2$, and thus equal to $k+2$, for every maximal relevant $(k+1)$-subcomplex $S\subseteq\fcc$. Let $S$ be a fixed such subcomplex, with $\vsupp(S)=\{v_1,\ldots, v_{k+2}\}$. Then $S$ consists of a single $(k+1)$-face which contains all of the vertices, $\sigma=\{v_1,\ldots,v_{k+2}\}$. So $\overline{S}=S$, and since $S$ is maximal, $\sigma$ is also a maximal face in $\fcc$. 

By Lemma \ref{int:dim}, the $k$-faces $\sigma_i=\sigma-\{v_i\}$ are not contained in $(k+1)$-faces in any other maximal relevant subcomplexes, and are thus all free faces of $\sigma$ in $\fcc$. So, without loss of generality, we perform the collapse $[\sigma_1,\sigma]$. This reduces $S$ to a $k$-dimensional subcomplex, and, again by Lemma \ref{int:dim}, is independent of collapses performed in any other maximal relevant $(k+1)$-subcomplex. Only one such collapse is required for each of the finitely many such subcomplexes. Therefore $\fcc$ is $(k+1)$-collapsible.\\
 
 {\bf Inductive Step for $l$:} 
Fix $l \ge k+3$. Assume that any finite clique complex is $(k+1)$-collapsible if $|\vsupp(S)|\le l-1$ for each of its maximal relevant $(k+1)$-subcomplexes. Then suppose $\fcc$ is such that each of its maximal relevant $(k+1)$-subcomplexes is supported on at most $l$ vertices, and let $S$ be a maximal relevant $(k+1)$-subcomplex on exactly $l$ vertices.

By assumption we have some $v \in S$ such that $\deg_{S}(v) \leq 2k+1$. So $\lk_{\overline{S}}(v)$ is a clique complex with at most $2k+1$ vertices, and thus its 1-skeleton is isomorphic either to the complete graph $K_{2k+1}$, or to a proper subgraph of $K_{2k+1}$. 

{\bf Case 1.} Suppose the 1-skeleton of $\lk_{\overline{S}}(v)$ is a proper subgraph of $K_{2k+1}$. Since $K_{2k+1}$ has uniform vertex degree $2k$, every induced subgraph, and hence every subcomplex of $\lk_{\overline{S}}(v)$, contains a vertex of degree at most $2k-1$. Hence $\lk_{\overline{S}}(v)$ satisfies the hypothesis of the theorem for $k-1$. By induction on $k$, $\lk_{\overline{S}}(v)$ is $k$-collapsible.

{\bf Case 2.} Now suppose the $1$-skeleton of $\lk_{\overline{S}}(v)$ is the complete graph $K_{2k+1}$. This does not satisfy the degree condition, but $\fcc$ is a clique complex, so $\lk_{\overline{S}}(v)$ is a $(2k)$-dimensional simplex. Label the vertices in the link $x,v_1,\ldots,v_{2k}$. One may consider the $(2k)$-dimensional simplex to be a cone over the $(2k-1)$-dimensional simplex with cone point $x$. So for a subset $\sigma\subseteq\{v_1,\ldots,v_{2k}\}$, collapses of the form $[\sigma,\sigma \cup \{ x \}]$ are sufficient to collapse $\lk_{\overline{S}}(v)$ to a complex of dimension at most $k-1$, and thus it is $k$-collapsible. 

We see that in both cases $\lk_{\overline{S}}(v)$ is $k$-collapsible, and so by Lemma \ref{lk:nicely}, $\fcc$ collapses to a complex in which all faces in $\st_{\overline{S}}(v)$ have dimension at most $k$. We then narrow our attention to $S'=\overline{S}-\st_{\overline{S}}(v)$. Since faces were only removed from $\st_{\overline{S}}(v)$, $S'$ is a finite clique complex with $\vsupp(S')=\vsupp(S)-\{v\}$. The $(k+1)$-skeleton of $S'$ may not be strongly connected, but all remaining high-dimensional faces from $\overline{S}$ are now partitioned by the maximal relevant $(k+1)$-subcomplexes of $S'$. And $|\vsupp(S')|\le l-1$, so all of it's maximal relevant $(k+1)$-subcomplexes are supported on at most $l-1$ vertices. All of the relevant $(k+1)$-subcomplexes of $S'$ are also trivially relevant $(k+1)$-subcomplexes of the original complex $\fcc$, and so by assumption they satisfy the degree condition. Thus $S'$ is a finite clique complex satisfying the degree condition, and by the inductive assumption on $l$, $S'$ is $(k+1)$-collapsible.

We can then apply the same process to each of the finitely many maximal relevant $(k+1)$-subcomplexes of $\fcc$, and thus $\fcc$ is $(k+1)$-collapsible. Finally, by induction, the theorem holds for any fixed $k\ge 0$.
\end{proof}

\begin{lem} \label{lem:coll_prob}
Fix $k \geq 0$ and $\epsilon > 0$, and let $$p \leq n^{-1/(k+1)-\epsilon}.$$ If $\fcc \dist X(n,p)$, then with high probability, every strongly connected, pure $(k+1)$-dimensional subcomplex of $\fcc$ has a vertex of degree at most $2k+1$.
\end{lem}

\begin{proof}
Lemma \ref{lem:coll_prob} appears as a key argument in the proof of Theorem 3.6 in ~\cite{clique} to show vanishing homology. We briefly summarize that proof here, for completeness. Let $N$ be any integer such that $N > 1 / \epsilon$. As observed in Lemma 5.1 of \cite{clique}, with high probability every strongly connected, pure $(k+1)$-dimensional subcomplex of $\fcc$ is supported on at most $N+k+2$ vertices.

Fix any graph $H$ on $\ell$ vertices with $\ell \le N+k+2$, and with minimum degree at least $2k+2$. Applying a union bound, an upper bound on the probability of $H$ being an induced subgraph of our random graph is
$$ {n \choose \ell} \ell! \, p^{(2k+2) \ell/2} \le n^\ell p^{(k+1) \ell}.$$
By the assumption on $p$, this function tends to zero as $n \to \infty$. As there are only finitely many isomorphism types of graphs on at most $N+k+2$ vertices, we can apply a union bound over all of them.
\end{proof}

\begin{pfofthm}{\ref{thm:collapse}}
By Lemma \ref{lem:coll_prob}, $\fcc\dist X(n,p)$ satisfies the degree condition for strongly connected, pure $(k+1)$-dimensional subcomplexes with high probability, and therefore Theorem \ref{thm:collapse} follows from Theorem \ref{thm:colla}.
\end{pfofthm}
\vskip 1cm

\section{The Bouquet of Spheres Conjecture}\label{kspheres}

There is a great deal of inherent interest in collapsibility thresholds in random models, however the primary motivation for this work is to leverage $(k+1)$-collapsibility in $X(n,p)$ as a step towards Kahle's ``Bouquet of Spheres" conjecture. Specifically, Kahle conjectured that if $\fcc\dist X(n,p)$ for $p=n^{-\alpha}$ with $1/(k+1)<\alpha <1/k$ for a fixed $k \ge 3$, then $\fcc$ is homotopy equivalent to a bouquet of $k$-spheres with high probability. Since the dimension of $\fcc$ is either $2k$ or $2k+1$ in this range, this would also imply that all the nontrivial homology is in middle dimension, and it would provide a probabilistic explanation for the overwhelming prevalence of bouquets of spheres in the study of combinatorial topology.

For $k=2$, it is known that $\fcc$ is not homotopy equivalent to a bouquet of 2-spheres, as Babson ~\cite{babson}, and independently Costa, Farber, and Horak ~\cite{CFH15}, showed that with high probability $\pi_1(\fcc)$ is a nontrivial hyperbolic group for $1/3 < \alpha < 1/2$, and thus $\fcc$ is not simply connected. For $k=1$, however, a connected 1-dimensional complex with $H_{1}(\fcc;\Z)\neq 0$ is known to be homotopy equivalent to a bouquet of 1-spheres. In their classic paper, Erd\H{o}s and R\'{e}nyi showed that the threshold for connectivity of a graph sampled from $G(n,p)$ is only $p=\log n/n$~\cite{ER}, so $\fcc$ is connected with high probability for $p\ge n^{-1/2+\epsilon}$. And by Corallary 3.7 in ~\cite{clique}, $H_1(\fcc, \Z)\neq 0$ in this range. Therefore, 2-collapsibility for $1/3 < \alpha < 1/2$ implies that $\fcc$ is indeed homotopy equivalent to a bouquet of 1-spheres with high probability. 

For $k\ge3$, the question remains open. However, collapsibility is a significant step in this direction. A standard method for proving that a topological space is homotopy equivalent to a bouquet of $k$-spheres is by showing that it is both $k$-dimensional and topologically $(k-1)$-connected, i.e. $\pi_j(\fcc)=0$ for $j\leq k-1$. By the Hurewicz Theorem, instead of showing that $\pi_j(\fcc)=0$, it is equivalent to show that $\fcc$ is simply connected and $H_j(\fcc; \Z)=0$ for $2\le j\le k-1$. Through a series of theorems in ~\cite{clique} and ~\cite{kahlesharp}, Kahle was able to prove that for $p$ in the appropriate range, $\fcc\dist X(n,p)$ is at least \emph{rationally} homotopy equivalent to a bouquet of $k$-spheres with high probability (Corollary 1.3 in ~\cite{kahlesharp}). Among these results, he showed that $\fcc$ is simply connected if $\alpha < 1/3$. Therefore, since Theorem \ref{thm:collapse} proves that $\fcc$ is homotopy equivalent to a $k$-dimensional complex with high probability, to prove the conjecture in full it only remains to show that if $p\gg n^{-1/(k+1)}$ for fixed $k\ge 3$, then $H_k(\fcc;\Z)=0$.
\\

\section{Acknowledgements} 
The author thanks Chris Fowler, Chris Hoffman, and Matt Kahle for many helpful discussions, and also David Sivakoff for catching a technical error in an early draft.
\\

\bibliographystyle{abbrv}
\bibliography{collapsebib}

\end{document}